\documentclass{amsart}

\usepackage{amssymb}
\usepackage{tikz}
\usetikzlibrary{positioning}

\newtheorem{theorem}{Theorem}[section]
\newtheorem{lemma}[theorem]{Lemma}
\newtheorem{proposition}[theorem]{Proposition}
\newtheorem{corollary}[theorem]{Corollary}

\theoremstyle{definition}
\newtheorem{definition}[theorem]{Definition}
\newtheorem*{example}{Example}

\theoremstyle{remark}
\newtheorem{claim}{Claim}

\numberwithin{equation}{section}
\makeatletter
  \@addtoreset{claim}{theorem}
  \@addtoreset{claim}{lemma}
  \@addtoreset{claim}{proposition}
\makeatother

\let\define\emph
\let\al\mathbf
\let\tup\mathbf
\let\join\vee
\let\meet\wedge

\let\bigjoin\bigvee

\let\pmod\bmod

\let\equals\approx
\let\midbar\mid

\DeclareMathOperator{\face}{face}
\DeclareMathOperator{\cube}{cube}

\DeclareMathOperator{\Cg}{Cg}
\DeclareMathOperator{\Sg}{Sg}
\DeclareMathOperator{\Clo}{Clo}
\DeclareMathOperator{\Pol}{Pol}
\DeclareMathOperator{\Con}{Con}

\author[J. Opr\v{s}al]{Jakub Opr\v{s}al}
\address{Department of Algebra, Faculty of Mathematics and Physics,
  Charles University in~Prague, Sokolovsk\'a 83, 186 75 Praha, Czech
  Republic}
\address{Institute for Algebra, Johannes Kepler University, 4040 Linz,
  Austria}
\email{oprsal@karlin.mff.cuni.cz}

\title[Relational description of higher commutators]{A~relational description
  of higher commutators in Mal'cev varieties}
\date{\today}



\thanks{Supported by Austrian Science Fund (FWF), P24077:
  Algebraic approaches to the description of~Mal'cev clones; Czech Science
  Foundation, GA\v{C}R 13-01832S; and Charles University in Prague, project
  SVV-2014-260107. }

\begin{document}

\begin{abstract}
  We give a~relational description of higher commutator operators, which were introduced by Bulatov, in varieties with a~Mal'cev term. Furthermore, we use this result to prove that for every algebra with a~Mal'cev term there exists a~largest clone containing the Mal'cev operation and having the same congruence lattice and the same higher commutator operators as the original algebra. We also give a~local variant of this theorem.
\end{abstract}

\maketitle

\section{Introduction}

Two algebras are called \emph{polynomially equivalent} if they have the same
underlying set and the same clone of all polynomial operations.  One of the
invariants to distinguish polynomially inequivalent algebras is the congruence
lattice of the corresponding algebra, and the binary commutator operation
$[{\cdot},{\cdot}]$ on this lattice (the theory describing this commutator have
been developed in the 80's, and is described in the book by Freese and McKenzie
\cite{fm87}). In fact, from the results of Idziak \cite{idz99} and Bulatov
\cite{bul05}, one can see that on the~three-element set, every Mal'cev algebra
is up to polynomial equivalence described by its congruence lattice, and the
binary commutator operation. This is no longer true for sets with at least
four elements.  But one can generalize the binary commutator operator to higher
arities. These higher arity commutators have been introduced by Bulatov
\cite{bul01}.  From the description of polynomial clones with a~Mal'cev
operation on the four-element set
\cite{bi03}, one can obtain that every four-element Mal'cev algebra is
determined up to polynomial equivalence by its unary polynomials, congruence
lattice, and higher commutator operators on this lattice. 
The higher commutators are defined by the following `term-condition'.

\begin{definition}[Bulatov's higher commutator operators] \label{def:bulatov}
Let $\alpha_0,\dots,\alpha_{n-1}$, and $\gamma$ be congruences of some algebra
$\al A$. We say that \define{$\alpha_0,\dots,\alpha_{n-2}$ centralize
$\alpha_{n-1}$ modulo $\gamma$} if for all tuples $\tup a_i$, $\tup b_i$,
$i=0,\dots,{n-1}$, and all terms $t$ of~$\al A$ such that
\begin{enumerate}
\item $\tup a_i \neq \tup b_i$, but the corresponding entries are congruent
modulo $\alpha_i$ for all $i \in \{0,\dots,n-1\}$, and
\item $t(\tup x_0,\dots,\tup x_{n-2}, \tup a_{n-1}) \equiv_\gamma t(\tup
x_0,\dots,\tup x_{n-2}, \tup b_{n-1})$ for all $(\tup
x_0,\dots,\tup x_{n-2}) \in (\{\tup a_0,\tup b_0\} \times \dots \times \{\tup
a_{n-2},\tup b_{n-2}\}) \setminus \{(\tup b_0,\dots,\tup b_{n-2})\}$,
\end{enumerate}
we have
\[
  t(\tup b_0,\dots,\tup b_{n-2},\tup a_{n-1}) \equiv_\gamma
    t(\tup b_0,\dots,\tup b_{n-2},\tup b_{n-1}).
\]
The $n$-ary \define{commutator} $[\alpha_0,\dots,\alpha_{n-1}]$ is then
defined as the~smallest congruence $\gamma$ such that
$\alpha_0,\dots,\alpha_{n-2}$ centralize $\alpha_{n-1}$ modulo $\gamma$.
We define the nulary commutator to be trivially the full congruence on $\al A$,
and for the unary commutator of $\alpha$ we put $[\alpha] = \alpha$.
\end{definition}

One of important notions that came from higher commutators is a~notion
of supernilpotence: an~algebra is $k$-supernilpotent if it satisfies
the commutator identity
\[
  [\underbrace{1,1,\dots,1}_{k+1}] = 0.
\]
If an algebra is $k$-supernilpotent for some $k$ we say that is it
supernilpotent.  For general algebras supernilpotence is a~strictly stronger
notion then nilpotence; i.e., there is a~nilpotent algebra which is not
supernilpotent.  However, this is not the case in the variety of groups where
both notions coincide. Therefore supernilpotent algebras can be viewed as
~natural generalization of nilpotent groups. They also share several properties
with nilpotent groups, in particular a~Mal'cev algebra of finite type is
supernilpotent if and only if it is a~product of prime power order
supernilpotent algebras \cite{am10}.  It has been shown in \cite{am13a} that
there are two expansions of the same group that are both 2-supernilpotent, but
the clone given as the join of their clones is not. In this paper we establish
additional properties to ensure that the join of two $k$-supernilpotent clones
sharing a~Mal'cev operation is $k$-supernilpotent.

To achieve that goal we give a~description of higher commutators using a~certain
$2^n$-ary relation denoted $\Delta(\alpha_0,\dots,\alpha_{n-1})$ (see
Definition~\ref{def:delta}). A~similar relation have been also
defined in \cite{sha14}. The relation $\Delta(\alpha_0,\dots,\alpha_{n-1})$
encodes the value of~$[\alpha_0,\dots,\alpha_{n-1}]$ as its \emph{forks} at the
last coordinate---by a~\emph{fork} of a~relation $R$ at a~coordinate $i$ we mean
a~pair $(a,b)$ such that there exists $\tup c, \tup d \in R$ with $c_i = a$,
$d_i = b$, and $c_j = d_j$ for all $j\neq i$; and we denote $\psi_i(R)$, the set
of all forks of $R$ at $i$. A similar notion has been used to investigate some
properties of algebras with a~cube term \cite{bim+10,amm14}. The
description of higher commutators is then given by the following theorem.

\begin{theorem} \label{description-of-higher-commutators}
  \label{commutators-are-forks}
If $\al A$ is an algebra with a~Mal'cev term, and $\alpha_0,\dots,\alpha_{n-1}$
are congruences of~$\al A$ then
\[
  [\alpha_0,\dots,\alpha_{n-1}] =
    \psi_{2^n-1}(\Delta(\alpha_0,\dots,\alpha_{n-1})).
\]
\end{theorem}

Further we show that $\Delta(\alpha_0,\dots,\alpha_{n-1})$ encodes not only the
commutator $[\alpha_0,\dots,\alpha_{n-1}]$ but also all smaller-arity
commutators that can be obtained by omitting one or more of the congruences
$\alpha_i$. We show that if we take the clone of all polymorphisms of the
relation $\Delta(\alpha_0,\dots,\alpha_{n-1})$ we get exactly the clone
$\mathcal C(\alpha_0,\dots,\alpha_{n-1})$ with the properties described in the
following theorem, and consequently one can construct a~largest clone with the
same commutator operators as the original Mal'cev algebra.

\begin{theorem} \label{functions-preserving-n-ary-commutator}
Let $\al A$ be an~algebra with Mal'cev term $q$, and let
$\alpha_0,\dots,\alpha_{n-1}$ be congruences of~$\al A$. Then there exists
a~largest clone $\mathcal C(\alpha_0,\dots,\alpha_{n-1})$ on~$A$ containing~$q$
such that it preserves congruences $\alpha_0,\dots,\alpha_{n-1}$, and all
commutators of the form $[\alpha_{i_0},\dots,\alpha_{i_{k-1}}]$ (where $k\leq n$
and $0\leq i_0<\dots<i_{k-1}<n$) agree in $\al A$ and $(A,\mathcal
C(\alpha_0,\dots,\alpha_{n-1}))$.
\end{theorem}

\begin{corollary} \label{functions-preserving-commutators-form-a-clone}
Let $\al A$ be an~algebra with a~Mal'cev term $q$, then there exists a~largest
clone on $A$ containing $q$ such that the algebra corresponding to this clone
has the same congruence lattice as $\al A$ and the same higher commutator
operators as~$\al A$.
\end{corollary}

\begin{proof}[Proof of Corollary
\ref{functions-preserving-commutators-form-a-clone} given Theorem
\ref{functions-preserving-n-ary-commutator}]
The largest such clone is the intersection of all clones $\mathcal
C(\alpha_1,\dots,\alpha_n)$ from Theorem
\ref{functions-preserving-n-ary-commutator} for all $n$ and all tuples 
$\alpha_1,\dots,\alpha_n$ of congruences of~$\al A$.
\end{proof}

Although our main motivation of developing this theory lies in the application
to Mal'cev algebras on a~finite domain, the same results are valid even for
algebras with infinite domains. Moreover, since the largest clone in the
previous theorem is described as a~polymorphism clone, we know that such clone
is closed in the natural topology given by pointwise convergence by a~result of
Romov \cite{rom77}. More on clones on infinite sets can be found in \cite{gp08}.

The theory developed to prove Theorem
\ref{functions-preserving-n-ary-commutator} is strong enough to give relatively
short proofs of several basic properties of higher commutators (usually referred
as (HC1)--(HC8)) that have been established in \cite{bul01}, their proofs have
been published in~\cite{am10}.  Our alternative proofs of some of these
properties are given in the last section of~this paper.
 \section{Preliminaries and notation}

Algebras are denoted by bold letters, the underlying set of an algebra is
denoted by the same letter in italic, $\Con \al A$ denotes the set of all
congruences of an~algebra $\al A$, $\Clo \al A$ the set (clone) of all term
operations of~$\al A$, $\Cg X$ denotes the congruence generated by $X$, $\Sg Y$
denotes a~subalgebra generated by $Y$, and if $\alpha$ is a~congruence then we
use the symbol $a \equiv_\alpha b$ to denote $(a,b) \in \alpha$. Furthermore, if
$R$ is a~relation, we use symbol $\Pol R$ to denote the clone of all polymorphisms
of~$R$. The~symbol $2$ will denote both the natural number $2$ and the two-element
set $\{0,1\}$.

We denote tuples by bold letters. The $i$-th coordinate of tuple $\tup a$ is
denoted by either $a_i$, or $\tup a{(i)}$.  So, $\tup a = (a_0,\dots,a_{n-1})$
and $(a_0,\dots,a_{n-1}){(i)} = a_i$.  Tuples will be usually numbered by
an increasing sequence of consecutive integers starting at $0$. So every $n$-ary
relation is a~subset of $A^{\{0,\dots,n-1\}} = A^n$. The only exception will be
elements of the~relation $\Delta(\alpha_0,\dots,\alpha_{n-1})$.  In the theory
of binary commutator described in \cite{fm87}, it is usual to denote the
elements of $4$-ary relation $\Delta_{\alpha,\beta}$ (we will denote the same
relation $\Delta(\alpha,\beta)$) as $2\times 2$ matrices. Similarly, when
generalizing this concept to $\Delta(\alpha_0,\dots,\alpha_{n-1})$ one should
write elements of this relation as $2 \times \dots \times 2$ $n$-dimensional
matrices.  We will denote those elements by tuples whose coordinates will
be labeled by the set $2^n = \{0,1\}^n$, we will write these coordinates as
binary sequences omitting brackets and commas, and if needed we will view them as
reverse binary expansions of natural numbers $0, \dots, 2^n-1$; i.e., the tuple
$\tup k = k_0\dots k_{n-1}$ represents the number $\sum k_i 2^i$. This gives us
a~natural linear ordering of the set $2^n$ that we will use to write the
elements of $A^{2^n}$ as linear $2^n$-tuples. So, the tuple $\tup a \in A^{2^n}$
will be written as
\(
  (a_{00\dots 0},a_{10\dots0},a_{010\dots 0}, a_{110\dots 0},
    \dots, a_{11\dots 1})
\). For $d\in\{0,1\}$ we will use symbol $\overline{d}$ for the negation of $d$,
i.e., $\overline 0 = 1$ and $\overline 1 = 0$. We will also refer to forks of
these relations at some~coordinate $\tup k$ the same way as if all the
coordinates would be integer.

The last piece of notation has a~close connection to a~simple lemma about forks of
a~relation. For any map $e\colon J \to I$ and $\tup a \in A^I$ the symbol $\tup
a^e$ denotes the~$J$-tuple defined by $\tup a^e(j) = a_{e(j)}$.  Similarly, for
a~relation $R\leq \al A^I$, $R^e$ denotes the relation $\{ \tup a^e \midbar \tup
a\in R \}$.

\begin{lemma} \label{forks-of-projections} 
Let $\al A$ be an~algebra, $R\leq \al A^I$, $S\leq \al A^J$, $e\colon J \to I$,
and $R^e \subseteq S$. If $i\in I$ and there is a~unique $j\in J$ such that
$e(j) = i$ then
$\psi_{i} (R) \subseteq \psi_{j} (S)$. In particular,
\begin{enumerate}
  \item[(i)] if $R\subseteq S$ then $\psi_i(R) \subseteq \psi_i(S)$ for every
  $i\in I$;
  \item[(ii)] if $e\colon I\to I$ is bijective then $\psi_{e(i)} (R) = \psi_i
  (R^e)$ for every $i\in I$.
\end{enumerate}
\end{lemma}

\begin{proof} Suppose that $(a,b) \in \psi_i(R)$; i.e., there are tuples $\tup
a$, $\tup b \in R$ such that $a_i = a$, $b_i = b$, and $a_k = b_k$ for all
$k\neq i$. Then from $R^e\subseteq S$ we know that $\tup a^e$, $\tup
b^e\in S$. These tuples witness that $(a,b) \in \psi_j (S)$, because
$\tup a^e(j) = a_{i} = a$, $\tup b^e(j) = b_{i} = b$,
and $\tup a^e(k) = \tup b^e(k)$
for $k\neq j$.

The statement (i) is a~special case of the former for $I=J$, and $e$ being the
identity mapping. To prove (ii), suppose that $e$ is a~bijection on the set
$I$. Then from the statement for $S = R^e$ we get that
\(
  \psi_{e(i)} (R) \subseteq \psi_{i} (R^e).
\)
For the other inclusion substitute $e$ with $e^{-1}$, $R$ with $R^e$, and $i$
with $e(i)$.
\end{proof}

We recall two simple well-known lemmata for Mal'cev algebras.

\begin{lemma} \label{reflexive-relation}
  Let $\al A$ be a~Mal'cev algebra. Then any binary reflexive compatible
  relation on $\al A$ is a~congruence. \qed
\end{lemma}

\begin{lemma} \label{forks}
Let $\al A$ be a~Mal'cev algebra, and let $R$ be $n$-ary compatible relation
on $A$. If $(a, b) \in \psi_i(R)$, and
$(c_0,\dots,c_{i-1},a,c_{i+1},\dots,c_{n-1}) \in R$ then
\[ (c_0,\dots,c_{i-1},b,c_{i+1},\dots,c_{n-1}) \in R. \]
\end{lemma}

\begin{proof}
Without loss of generality suppose that $i=0$. Let $q$ be a~Mal'cev term of
$\al A$, and let $(a,u_1,\dots,u_{n-1})$ and $(b,u_1,\dots,u_{n-1})$ be
witnesses for $(a,b) \in \psi_0(R)$. Then
\[
q \begin{pmatrix}
    b & a & a \\
    u_1 & u_1 & c_1 \\
    \vdots && \vdots \\
    u_{n-1} & u_{n-1} & c_{n-1} \\
  \end{pmatrix}
  = \begin{pmatrix} b \\ c_1 \\ \vdots \\ c_{n-1} \end{pmatrix}
    \in R,
\]
since we know that $R$ is compatible with $q$.
\end{proof}
 \section{Description of higher commutators} \label{the-description}

\begin{definition} \label{def:delta}
Let $\al A$ be an~algebra, and $\alpha_0,\dots,\alpha_{n-1} \in \Con \al A$.
First, for each congruence $\alpha_i$ choose one dimension in the
$n$-dimensional space. We define the relation $\Delta_{\al A}
(\alpha_0,\dots,\alpha_{n-1})$ as the $2^n$-ary relation indexed by the set
$2^n$ generated by tuples that are constant on two opposing $(n-1)$-dimensional
hyperfaces of~the hypercube orthogonal to~the dimension corresponding to
$\alpha_i$ and these constants are $\alpha_i$ congruent. 

We will use $\face_i^d (\tup a)$ to denote the $(d+1)$-th hyperface orthogonal to
dimension $i$, i.e., $\face_i^d(\tup a) = \tup a^{f_{i,d}}$ where $f_{i,d}(\tup
k) = k_0\dots k_{i-1}dk_i\dots k_{n-2}$.  The generating tuples of the relation
$\Delta_{\al A}(\alpha_0,\dots,\alpha_{n-1})$ will be denoted $\cube_i^n(a,b)$.
By definition, \( \face_i^0\cube_i^n(a,b) = (a,\dots,a) \), and \(
\face_i^1\cube_i^n(a,b) = (b,\dots,b) \); or equivalently, $\cube_i^n (a,b) =
(a,b)^{d_i}$ where $d_i\colon 2^n \to 2$ is defined by $\tup k \mapsto k_{(i)}$.
Finally, $\Delta_{\al A}(\alpha_0,\dots,\alpha_{n-1})$ is defined as
\[
\Delta_{\al A}(\alpha_0,\dots,\alpha_{n-1}) :=
  \Sg \big\{  \cube_i^n(a,b)\midbar i<n, a\equiv_{\alpha_i} b \big\} =
  \bigjoin_{i<n} \cube_i^n (\alpha_i).
\]
For the trivial case when $n=0$, we put $\Delta_{\al A} () := A$.
If the algebra is clear from the context, we will write just
$\Delta(\alpha_0,\dots,\alpha_{n-1})$ instead of~$\Delta_{\al A}
(\alpha_0,\dots,\alpha_{n-1})$, and if $\mathcal C$ is a~clone on the~set $A$,
we will write $\Delta_{\mathcal C} (\alpha_0,\dots,\alpha_{n-1})$ instead of
$\Delta_{(A,\mathcal C)} (\alpha_0,\dots,\alpha_{n-1})$.
\end{definition}

\begin{example} We will describe generators
of~$\Delta(\alpha_0,\alpha_1,\alpha_2)$ for three congruences $\alpha_0$,
$\alpha_1$, $\alpha_2$ of an algebra~$\al A$. The elements
of~$\Delta(\alpha_0,\alpha_1,\alpha_2)$ are indexed by vertices of
a~three-dimensional hypercube.
The generators are tuples of one of the following forms $(a,b,a,b,a,b,a,b)$,
where $a \equiv_{\alpha_0} b$, $(a,a,b,b,a,a,b,b)$ where $a \equiv_{\alpha_1}
b$, and $(a,a,a,a,b,b,b,b)$ where $a \equiv_{\alpha_2} b$. Their graphical
representation is given in Figure~\ref{delta-3-generators}.
\begin{figure}[ht]
\centering
\def\size{1.5} \def\shift{3.2}
\begin{tikzpicture}[ every node/.style={circle,inner sep=1} ]
\node (A) at (-\shift,0,0)             {$a$};
\node (B) at (-\shift,0,\size)         {$b$};
\node (C) at (-\shift,\size,0)         {$a$};
\node (D) at (-\shift,\size,\size)     {$b$};
\node (E) at (\size-\shift,0,0)         {$a$};
\node (F) at (\size-\shift,0,\size)     {$b$};
\node (G) at (\size-\shift,\size,0)     {$a$};
\node (H) at (\size-\shift,\size,\size) {$b$};

\draw  (A)--(C)--(G)--(E)--(A);
\draw  (B)--(D)--(H)--(F)--(B);
\draw [ultra thick] (A)--(B);
\draw [ultra thick] (C)--(D);
\draw [ultra thick] (E)--(F) node [midway,below right=.05] {\small $\alpha_0$};
\draw [ultra thick] (G)--(H);

\node (A) at (0,0,0)             {$a$};
\node (B) at (0,0,\size)         {$a$};
\node (C) at (0,\size,0)         {$b$};
\node (D) at (0,\size,\size)     {$b$};
\node (E) at (\size,0,0)         {$a$};
\node (F) at (\size,0,\size)     {$a$};
\node (G) at (\size,\size,0)     {$b$};
\node (H) at (\size,\size,\size) {$b$};

\draw  (A)--(B)--(F)--(E)--(A);
\draw  (C)--(D)--(H)--(G)--(C);

\draw [ultra thick] (A)--(C);
\draw [ultra thick] (B)--(D);
\draw [ultra thick] (F)--(H);
\draw [ultra thick] (E)--(G) node [midway,right=.05] {\small $\alpha_1$};

\node (A) at (\shift,0,0)             {$a$};
\node (B) at (\shift,0,\size)         {$a$};
\node (C) at (\shift,\size,0)         {$a$};
\node (D) at (\shift,\size,\size)     {$a$};
\node (E) at (\size+\shift,0,0)         {$b$};
\node (F) at (\size+\shift,0,\size)     {$b$};
\node (G) at (\size+\shift,\size,0)     {$b$};
\node (H) at (\size+\shift,\size,\size) {$b$};

\draw (A)--(B)--(D)--(C)--(A);
\draw (E)--(F)--(H)--(G)--(E);
\draw [ultra thick] (A)--(E);
\draw [ultra thick] (B)--(F) node [midway,below] {\small $\alpha_2$};
\draw [ultra thick] (C)--(G);
\draw [ultra thick] (D)--(H);

\end{tikzpicture}
\caption{Generators of $\Delta(\alpha_0,\alpha_1,\alpha_2)$} \label{delta-3-generators}
\end{figure}
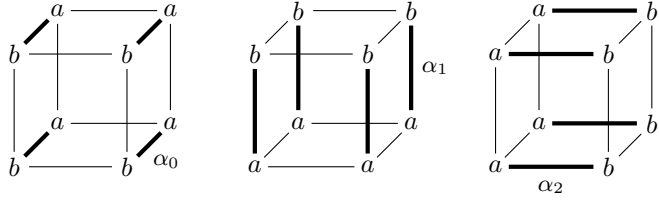
\end{example}

Before we get to the proof of Theorem~\ref{commutators-are-forks}, we will
describe some basic properties of the relation
$\Delta(\alpha_0,\dots,\alpha_{n-1})$.  The first lemma gives a~term
description of $\Delta(\alpha_0,\dots,\alpha_{n-1})$.  This description gives
a~clear connection of $\Delta(\alpha_0,\dots,\alpha_{n-1})$ and the term
condition.

\begin{lemma} \label{term-description-of-delta} 
For every algebra $\al A$, and congruences $\alpha_i \in\Con A$, $i<n$,
\begin{multline*}
\Delta(\alpha_0,\dots,\alpha_{n-1}) \\ = 
  \big\{ \big(
    t( \tup a_0, \dots, \tup a_{n-1} ),
    t( \tup b_0, \tup a_1 \dots, \tup a_{n-1} ),
    \dots,
    t( \tup b_0, \tup b_1 \dots, \tup b_{n-1} )
    \big) \midbar \\
  \forall i<n : m_i \in \mathbb N_0, \tup a_i,\tup b_i \in A^{m_i},
    \tup a_i \equiv_{\alpha_i} \tup b_i; t\in \Clo_{\sum m_i} \al A 
  \big\}
\end{multline*}
where the elements of the $2^n$-tuple include the term $t$ applied to all the
combinations of corresponding $\tup a_i$'s and $\tup b_i$'s.
\end{lemma}

\begin{proof}
  The relation $\Delta(\alpha_0,\dots,\alpha_{n-1})$ is generated by tuples
  $\cube_i^n(a,b)$ for $a \equiv_{\alpha_i} b$. So,
  $\Delta(\alpha_0,\dots,\alpha_{n-1})$ is the set of all tuples of the form
  \[
    t( \cube_{i_0}^n(a_0,b_0), \dots, \cube_{i_{k-1}}^n(a_{k-1},b_{k-1}) )
  \]
  where $t\in \Clo_k \al A$, and for all $j<k$ we have $i_j < n$, $a_j \equiv
  b_j$ modulo $\alpha_{i_j}$. The~description in the statement of the lemma can
  be obtained from this by grouping together $\cube_{i_j}^n(a_j,b_j)$'s with
  the same index $i_j$, and applying the term $t$ coordinatewise.
\end{proof}

\begin{example}
In the ternary commutator case, the lemma tells that
\begin{multline*}
  \Delta(\alpha_0,\dots,\alpha_{n-1}) =
  \big\{ \big(
    t(\tup a_0,\tup a_1,\tup a_2),
    t(\tup b_0,\tup a_1,\tup a_2),
    t(\tup a_0,\tup b_1,\tup a_2),
    t(\tup b_0,\tup b_1,\tup a_2), \\
    t(\tup a_0,\tup a_1,\tup b_2),
    t(\tup b_0,\tup a_1,\tup b_2),
    t(\tup a_0,\tup b_1,\tup b_2),
    t(\tup b_0,\tup b_1,\tup b_2)
  \big) \midbar\\
  m_0,m_1,m_2 \in \mathbb N_0,  t\in \Clo_{m_0+m_1+m_2} \al A,
    \forall i < 3: \tup a_i,\tup b_i \in A^{m_i},
      \tup a_i \equiv_{\alpha_i} \tup b_i \big\}.
\end{multline*}
The graphical representation of a~typical element of this relation is given
in Figure~\ref{fig:term-description-of-delta}.

\begin{figure}[ht]
\newcommand{\size}{2.6}
\newcommand{\vsiz}{2.4}
\begin{tikzpicture}[every node/.style={inner sep=1}]
\node (A) at (0,0,0)             {$t(\tup a_0,\tup a_1,\tup a_2)$};
\node (B) at (0,0,\size)         {$t(\tup b_0,\tup a_1,\tup a_2)$};
\node (C) at (0,\vsiz,0)         {$t(\tup a_0,\tup b_1,\tup a_2)$};
\node (D) at (0,\vsiz,\size)     {$t(\tup b_0,\tup b_1,\tup a_2)$};
\node (E) at (\size,0,0)         {$t(\tup a_0,\tup a_1,\tup b_2)$};
\node (F) at (\size,0,\size)     {$t(\tup b_0,\tup a_1,\tup b_2)$};
\node (G) at (\size,\vsiz,0)     {$t(\tup a_0,\tup b_1,\tup b_2)$};
\node (H) at (\size,\vsiz,\size) {$t(\tup b_0,\tup b_1,\tup b_2)$};

\draw  (A)--(C)--(G)--(E)--(A);
\draw  (B)--(D)--(H)--(F)--(B);
\draw  (A)--(B);
\draw  (C)--(D);
\draw  (E)--(F);
\draw  (G)--(H);
\end{tikzpicture}
\caption{} \label{fig:term-description-of-delta}
\end{figure}
\end{example}

\begin{lemma} \label{term-condition} 
Let $\al A$ be an~algebra, and $\alpha_0$, \dots, $\alpha_{n-1}$, $\gamma \in
\Con \al A$.  Then $\alpha_0,\dots,\alpha_{n-2}$ centralize $\alpha_{n-1}$
modulo $\gamma$ if for every tuple
\[
  (a_{0\dots0},\dots,a_{1\dots 1},b_{0\dots0},\dots,b_{1\dots 1}) \in
    \Delta(\alpha_0,\dots,\alpha_{n-1})
\]
such that $a_{\tup k}
\equiv_\gamma b_{\tup k}$ for all $\tup k\in 2^{n-1} \setminus
\{1\dots1\}$ we have also $a_{1\dots1} \equiv_\gamma b_{1\dots1}$. \qed
\end{lemma}

\begin{lemma} \label{delta-is-a-congruence} 
Let $\al A$ be an~algebra with congruences $\alpha_0,\dots,\alpha_n$, $i<n$, and
$s_i\colon k\mapsto {k_{(0)}\dots k_{(i-1)}\overline{k_{(i)}}k_{(i+1)}\dots
k_{(n-1)}}$. Then
\begin{enumerate}
\item[(i)]
  if $d\in\{0,1\}$ and\/ $\tup a \in \Delta(\alpha_0,\dots,\alpha_{n-1})$ then 
  \[
    \face_i^d \tup a \in
      \Delta(\alpha_0,\dots,\alpha_{i-1},\alpha_{i+1},\dots,\alpha_{n-1});
  \]
\item[(ii)] if $\tup a \in \al A^{2^{n-1}}$ such that
  \[
    \face_i^0 \tup a = \face_i^1 \tup a
      \in \Delta(\alpha_0,\dots,\alpha_{i-1},\alpha_{i+1},\dots,\alpha_{n-1}),
  \]
  then $\tup a \in \Delta(\alpha_{0},\dots,\alpha_{n-1})$;
\item[(iii)] 
  \(
    \Delta(\alpha_0,\dots,\alpha_{n-1})^{s_i}
      = \Delta(\alpha_0,\dots,\alpha_{n-1}).
  \)
\end{enumerate}
Furthermore, if $\al A$ is a~Mal'cev algebra then
the binary relation
\[
  \delta = \left\{ \left(\face_i^0 \tup a,\face_i^1 \tup a\right)
    \midbar \tup a \in \Delta(\alpha_0,\dots,\alpha_{n-1}) \right\}
\]
is a~congruence on
$\Delta(\alpha_0,\dots,\alpha_{i-1},\alpha_{i+1},\dots,\alpha_n)$.
\end{lemma}

\begin{proof}
Since all the considered relations are compatible, we can check the
validity of particular inclusions on the generators of the relations, i.e.,
tuples $\cube_{i,n}(a,b)$. In detail, to prove (i) one has to observe that 
\[
  \face_i^d \cube_j^n (a,b) = \begin{cases}
    \cube_j^{n-1}(a,b) & \text{if $j<i$,} \\
    (c,\dots,c) & \text{where $c \in \{a,b\}$ if $j=i$,} \\
    \cube_{j-1}^{n-1}(a,b) & \text{if $j>i$,} \\
  \end{cases}
\]
and consequently
\begin{multline*}
  \face_i^d
    \{ \cube_j^n(a,b) \midbar j<n, a\equiv_{\alpha_j} b \} \\
    = \{ \cube_j^{n-1}(a,b) \midbar j<i, a\equiv_{\alpha_j} b \}
    \cup \{ \cube_{j-1}^{n-1}(a,b) \midbar i<j<n , a\equiv_{\alpha_j} b \}.
\end{multline*}
Hence the relation generated by the left hand side is the same as the relation
generated by the right hand side which gives the desired equality.

For (ii), first observe that
\[
  \{ \tup a \midbar
    \face_i^0 \tup a = \face_i^1 \tup a 
    \in \Delta(\alpha_0,\dots,\alpha_{i-1},\alpha_{i+1},\dots,\alpha_{n-1})\}
\]
is the~$2^n$-ary relation generated by tuples $\tup a$ such that $\face_i^0 \tup
a = \face_i^1 \tup a = \cube_{j,n-1}(a,b)$, $a \equiv b$ modulo $\alpha_j$, or
$\alpha_{j+1}$ when $j<i$, or $j\geq i$ respectively.  Second, if $\tup a$ is
such a~tuple then $\tup a = \cube_j^n(a,b)$ if $j<i$, or $\tup a =
\cube_{j+1}^n(a,b)$ if $j\geq i$.  In either case, $\tup a \in
\Delta(\alpha_0,\dots,\alpha_{n-1})$ for $a$, $b$ that are congruent modulo the
corresponding congruence.

The item (iii) is a~consequence of the fact that
\[
  (\cube_j^n(a,b))^{s_i} = \begin{cases}
    \cube_j^n(b,a) &\text{if $j=i$, or} \\
    \cube_j^n(a,b) &\text{otherwise.}
  \end{cases}
\]
Hence $(\cube_j^n \alpha_j)^{s_i} = \cube_j^n\alpha_j$ from the
symmetry of congruences. From the definition of
$\Delta(\alpha_0,\dots,\alpha_{n-1})$, we get that
\[
  \Delta(\alpha_0,\dots,\alpha_{n-1})^{s_i} =
    \Delta(\alpha_0,\dots,\alpha_{n-1}).
\]

From items (i)--(iii) we already know that the binary relation $\delta$ is
a~reflexive symmetric binary relation on
$\Delta(\alpha_0,\dots,\alpha_{i-1},\alpha_{i+1},\dots,\alpha_{n-1})$. The
rest follows from Lemma~\ref{reflexive-relation}.
\end{proof}

\begin{lemma} \label{forks-of-delta}
Let $\al A$ be a~Mal'cev algebra with congruences
$\alpha_0,\dots,\alpha_{n-1}$. Then
$
  \psi_{\tup j}(\Delta(\alpha_0,\dots,\alpha_{n-1})) =
  \psi_{\tup k}(\Delta(\alpha_0,\dots,\alpha_{n-1}))
$
for all $\tup j,\tup k \in 2^n$.
\end{lemma}

\begin{proof}
For simplicity, let
$\psi_{\tup k} = \psi_{\tup k}(\Delta(\alpha_0,\dots,\alpha_{n-1}))$.
If $s_i$ is the permutation on $2^n$ defined by
$s_i(\tup k) = k_0\dots k_{i-1}\overline{k_i}k_{i+1}\dots k_{n-1}$
then from Lemma \ref{delta-is-a-congruence}(iii), we know that
\(
  \Delta(\alpha_0,\dots,\alpha_{n-1}) =
  \Delta(\alpha_0,\dots,\alpha_{n-1})^{s_i}
\).
This gives, by Lemma~\ref{forks-of-projections}, that $\psi_{\tup k} =
\psi_{s_i(\tup k)}$ for all $i<n$. But, if $i_0$, \dots, $i_{m-1}$ are exactly
those indices $i$ such that $k_{i} \neq j_{i}$ then
$\tup j = s_{i_0} \circ \cdots \circ s_{i_{m-1}} (\tup k)$, and consequently
$\psi_{\tup k} = \psi_{\tup j}$ for all $\tup j$, $\tup k$.
\end{proof}

Instead of proving Theorem \ref{description-of-higher-commutators} directly,
we will prove the following refinement. The theorem is then given by
equivalence of (1) and (4).

\begin{proposition} \label{commutator-and-delta} 
If $\al A$ is a~Mal'cev algebra, $\alpha_0,\dots,\alpha_{n-1} \in \Con \al A$,
and $a,b\in \al A$; then the following is equivalent
\begin{enumerate}
  \item $(a,b) \in {\psi_{1\dots1}(\Delta(\alpha_0,\dots,\alpha_{n-1}))}$;
  \item $(a,\dots,a,b) \in \Delta(\alpha_0,\dots,\alpha_{n-1})$;
  \item there exists $c_0,\dots,c_{2^{n-1}-2}$ such that
  \[
    (c_0,\dots,c_{2^{n-1}-2},a,c_0,\dots,c_{2^{n-1}-2},b)
      \in \Delta(\alpha_0,\dots,\alpha_{n-1}).
  \]
  \item $a \equiv_{[\alpha_0,\dots,\alpha_{n-1}]} b$;
\end{enumerate}
\end{proposition}

\begin{proof}
The implication $(1)\to(2)$ is direct corollary of Lemma \ref{forks}, given that
$\Delta(\alpha_0,\dots,\alpha_{n-1})$ contains all constant tuples, in
particular $(a,\dots,a)$. $(2)\to(3)$ is trivial.  For $(3)\to(4)$ suppose that
$(3)$ is satisfied for a~pair $(a,b)$ then since $c_i \equiv c_i \pmod
{[\alpha_0,\dots,\alpha_{n-1}]}$ for all $i < 2^{n-1} - 1$ and
$\alpha_0,\dots,\alpha_{n-2}$ centralize $\alpha_{n-1}$ modulo
$[\alpha_0,\dots,\alpha_{n-1}]$, we have $a\equiv b
\pmod {[\alpha_0,\dots,\alpha_{n-1}]}$ from Lemma~\ref{term-condition}.

The last to prove is $(4)\to(1)$, in other words that
\begin{equation} \label{eq:4to1}
  [\alpha_0,\dots,\alpha_{n-1}] \leq
  \psi_{1\dots1}(\Delta(\alpha_0,\dots,\alpha_{n-1})).
\end{equation}
Let $\psi = \psi_{1\dots1}(\Delta(\alpha_0,\dots,\alpha_{n-1}))$; from Lemma
\ref{forks-of-delta} we know that
\[
  \psi = \psi_{1\dots1}(\Delta(\alpha_0,\dots,\alpha_{n-1})) =
    \psi_{\tup k}(\Delta(\alpha_0,\dots,\alpha_{n-1}))
\]
for all $\tup k \in 2^n$, so we do not have to distinguish between forks at
different coordinates. To prove (\ref{eq:4to1}) by the definition of the
commutator, it is enough to prove that $\alpha_0,\dots,\alpha_{n-2}$
centralize $\alpha_{n-1}$ modulo~$\psi$.  For that we will use Lemma
\ref{term-condition}. Suppose that
\[
 ( a_{0\dots0}, \dots, a_{1\dots 1}, b_{0\dots 0}, \dots, b_{1\dots 1} )
\in \Delta(\alpha_0,\dots,\alpha_{n-1}),
\]
and $a_{\tup i} \equiv_\psi b_{\tup i}$  for all ${\tup i}\in 2^{n-1} \setminus
\{1\dots 1\}$.  By repeated use of Lemma
\ref{forks} we can replace $b_{00\dots 0}$, \dots, $b_{01\dots 1}$ by the respective
$a_{\tup i}$'s. Hence,
\[
  ( a_{00\dots 0}, \dots, a_{01\dots 1}, a_{11\dots 1},
    a_{00\dots 0}, \dots, a_{01\dots 1}, b_{11\dots 1} )
  \in \Delta(\alpha_0,\dots,\alpha_{n-1}).
\]
In addition, we know from Lemma \ref{delta-is-a-congruence}(i) and
\ref{delta-is-a-congruence}(ii) that 
\[
  ( a_{00\dots 0}, \dots, a_{01\dots 1}, a_{11\dots 1},
    a_{00\dots 0}, \dots, a_{01\dots 1}, a_{11\dots 1} )
    \in \Delta(\alpha_0,\dots,\alpha_{n-1}).
\]
So, $a_{11\dots 1} \equiv_\psi b_{11\dots 1}$ which concludes the proof that
$\alpha_0,\dots,\alpha_{n-2}$ centralize $\alpha_{n-1}$ modulo $\psi$.
\end{proof}

In the last proposition some parts have been already known. The proposition
(in the case of Mal'cev algebras) generalize Theorem 4.9 of~\cite{fm87} which
gives equivalence of (2), (3), and (4) for binary commutators in
congruence modular varieties. The omitted equivalence of (1) and (3) in the
binary case can be easily derived from the known fact that
$\Delta(\alpha_0,\alpha_1)$ is a~congruence on rows.  Furthermore, for the
higher commutators, the implication $(3)\to(4)$ for the variety of
groups has appeared in \cite{sha14}; and if all $\alpha_i$'s are principal
congruences, the equivalence of (2) and (4) is given by
\cite[Lemma~6.13]{am10} (via an easy translation similar
to~Lemma~\ref{term-description-of-delta}).
 \section{Strong cube terms, and clones of operations preserving commutators}

We will use terms that generalize Mal'cev terms. These terms
will play similar role as a~difference term in the case of binary
commutator.
A~$(2^n-1)$-ary term $q_n$ is a~\define{strong $n$-cube term} if it satisfies 
\[
  q_n(x_{00\dots 0},x_{10\dots0},\dots,x_{01\dots 1}) \equals x_{11\dots 1}
\]
whenever there is $i<n$ such that for all $\tup k$ we have $x_{\tup k} =
x_{k_0\dots k_{i-1}\overline{k_i}k_{i+1}\dots k_{n-1}}$. This gives a~set of $n$
identities, each with $2^{n-1}$ variables. The two identities for strong
$2$-cube term are
\[
  q_2(x,y,x) \equals y
    \quad\text{and}\quad 
  q_2(x,x,y) \equals y.
\]
So, the term $q_2(y,x,z)$ is a~Mal'cev term, and if $q$ is a~Mal'cev term then
$q(y,x,z)$ is a~strong $2$-cube term. The three identities for strong $3$-cube
term are
\begin{align*}
q_3(x,y,z,w,x,y,z) &\equals w \\
q_3(x,y,x,y,z,w,z) &\equals w \\
q_3(x,x,y,y,z,z,w) &\equals w.
\end{align*}

Strong cube terms are stricter version of cube terms introduced in
\cite{bim+10}; a~$(2^n-1)$-ary term is an~\emph{$n$-cube term} if it satisfies
\[
  q_n(x_{00\dots 0},x_{10\dots 0},\dots,x_{01\dots 1}) \equals x_{11\dots 1}
\]
whenever there is $i<n$ such that $x_{\tup k} = x_{\tup l}$ for all $\tup k,
\tup l$ such that $k_i = l_i$. This gives $n$ two-variable identities
compared to the $2^{n-1}$ variables used in the identities for strong cube
terms. One can see that for $n\geq 2$ every strong $n$-cube term satisfies the
identities of an~$n$-cube term; the $i$-th identity of a~cube term is implied by
almost any identity for a~strong cube term except for the $i$-th one.

\begin{lemma} \label{existence-of-strong-cube-terms} 
The following is equivalent for any algebra $\al A$.
\begin{enumerate}
  \item $\al A$ has a~strong $n$-cube term for all $n\geq 2$,
  \item $\al A$ has a~strong $n$-cube term for some $n\geq 2$,
  \item $\al A$ has a~Mal'cev term.
\end{enumerate}
\end{lemma}
\begin{proof}
From the observation in the above paragraphs, we know that the condition (3) is
equivalent to
\begin{enumerate}
  \item[(3')] $\al A$ has a~strong $2$-cube term.
\end{enumerate}
We will prove equivalence of (1), (2), and (3'). The implication $(1) \to (2)$
is trivial. For $(2)\to (3')$ observe that if $q_n$ is a~strong $n$-cube term
then
\( q_2(x,y,z) = q_n(x,\dots,x,x,y,z) \)
is a~strong $2$-cube term. For the last implication $(3') \to (1)$ we can
construct strong cube terms by the recursion:
\[
q_{n+1} (x_0,\dots,x_{2^{n+1}-1}) =
  q_2 (q_n ( x_0,\dots,x_{2^n -2} ),
    x_{2^n -1},
    q_n ( x_{2^n},\dots,x_{2^{n+1} -2} ))
\]
It is easy to check that if $q_n$ is a~strong $n$-cube term and $q_2$ is
a~strong $2$-cube term then $q_{n+1}$ is a~strong $(n+1)$-cube term.
\end{proof}

The following lemma is the key for proving Theorem
\ref{functions-preserving-n-ary-commutator} and a~lot of properties of higher
commutators in Mal'cev varieties.

\begin{lemma} \label{delta-contains-graph-of-cube-term} 
Let $\al A$ be an~algebra with a~strong $n$-cube term $q_n$,
$\alpha_1,\dots,\alpha_n \in \Con \al A$. Then $\tup a \in
\Delta(\alpha_1,\dots,\alpha_n)$ if and only if
\[
  \face_i^0 \tup a \in
  \Delta(\alpha_0,\dots,\alpha_{i-1},\alpha_{i+1},\dots,\alpha_{n-1})
\]
for each $i$ and 
\(
  q_n(a_{00\dots 0},a_{10\dots 0},\dots,a_{01\dots 1}) \equiv_{[\alpha_0,\dots,\alpha_{n-1}]}
  a_{11\dots 1}
\).
\end{lemma}
\begin{proof}
We will prove the lemma in two steps.

\begin{claim} 
If $\face_i^0 \tup a \in
\Delta(\alpha_0,\dots,\alpha_{i-1},\alpha_{i+1},\dots,\alpha_{n-1})$ for each
$i<n$, and $q_n(a_{00\dots 0},$ $a_{10\dots 0},\dots,a_{01\dots 1}) = a_{11\dots1}$
then $\tup a \in \Delta(\alpha_1,\dots,\alpha_{n-1});$
\end{claim}

For $\tup k \in 2^n$ and $\tup b \in A^{2^n}$ let $\tup a^{\tup k}$ denotes the
$2^n$-tuple
\[
  \tup a^{\tup k} =
    (a_{0 0\dots 0}, a_{k_0 0\dots 0}, a_{0 k_1 \dots 0},
      \dots, a_{k_0k_1\dots k_{n-1}}).
\]
Note that if $k_i = 0$ then $\face_i^0 \tup a^{\tup k} = \face_i^1 \tup a^{\tup
k} = \face_i^0 \tup a^{k_0\dots k_{i-1}1k_{i+1}\dots k_{n-1}}$. Suppose that the
tuple $\tup a$ satisfies the premise of the claim. The fact that $\tup a^{\tup k} \in
\Delta(\alpha_0,\dots,\alpha_{n-1})$ for all $\tup k \neq 11\dots 1$ follows
from induction on the number of $0$'s in $\tup k$---the base
step is given by Lemma~\ref{delta-is-a-congruence}(ii) and the assumption; the
induction step is given by Lemma~\ref{delta-is-a-congruence}(i) and
(ii) and the above observation.  Next we claim that
\begin{equation} \label{eq:key.1}
  q_n^{\al A\!^{2^n}}
    (\tup a^{00\dots 0},\dots,\tup a^{01\dots 1}) =
    (a_{00\dots 0},\dots,a_{01\dots 1},
      q_n^{\al A}(a_{00\dots 0},\dots,a_{01\dots 1}));
\end{equation}
i.e., we have to show that
\[
  q_n ( a_{00\dots 0}, a_{j_00\dots 0}, \dots, a_{0j_1\dots j_{n-1}} )
    = a_{\tup j}
\]
for every coordinate $\tup j \neq 11\dots 1$.  The above is trivial for the
coordinate $\tup j = 11\dots 1$;  for $\tup j$ with $j_i = 0$ the identity
follows from the $i$-th identity for a~strong cube term.  Finally, since the
relation $\Delta(\alpha_0,\dots,\alpha_{n-1})$ is compatible with $q_n$ we
know that the right hand side of~(\ref{eq:key.1}) is in
$\Delta(\alpha_0,\dots,\alpha_{n-1})$.

\begin{claim} 
If $\tup a \in \Delta(\alpha_0,\dots,\alpha_{n-1})$ then
$q_n(a_{00\dots 0},\dots,a_{01\dots 1})
\equiv_{[\alpha_0,\dots,\alpha_{n-1}]} a_{11\dots 1}$.
\end{claim}

If $(a_{00\dots 0},\dots,a_{11\dots 1}) \in \Delta(\alpha_0,\dots,\alpha_{n-1})$
then from Lemma~\ref{delta-is-a-congruence}(i) we know that the tuple
\(
  (a_{00\dots 0},\dots,a_{01\dots 1}, q_n(a_{00\dots0},\dots,a_{01\dots1}))
\)
satisfies the prerequisites of the first claim. Hence, we know that
\[
  (a_{00\dots 0},\dots,a_{01\dots 1}, q_n(a_{00\dots 0},\dots,a_{01\dots1}))
    \in \Delta(\alpha_0,\dots,\alpha_{n-1})
\]
which shows that
\[
  (a_{11\dots 1},q_n(a_{00\dots 0},\dots,a_{01\dots1})) \in
    \psi_{11\dots 1}(\Delta(\alpha_0,\dots,\alpha_{n-1})) =
    [\alpha_0,\dots,\alpha_{n-1}].
\]

Finally, we get to the statement of this lemma. The `only if' part is Claim~2
together with Lemma~\ref{delta-is-a-congruence}(i).  For the `if' part, if
\[
  \face_i^0 \tup a \in
  \Delta(\alpha_0,\dots,\alpha_{i-1},\alpha_{i+1},\dots,\alpha_{n-1})
\]
for all $i<n$, we know from Claim~1 that
\[
  (a_{00\dots 0},\dots,a_{01\dots 1}, q_n(a_{00\dots 0},\dots,a_{01\dots1}))
    \in \Delta(\alpha_0,\dots,\alpha_{n-1}).
\]
From the last condition and
Theorem~\ref{description-of-higher-commutators}, we know that
$q_n(a_{00\dots 0},\dots,a_{01\dots1})$ and $a_{11\dots 1}$
are congruent modulo ${\psi_{11\dots 1} (\Delta(\alpha_0,\dots,\alpha_{n-1}))}$.
Therefore, $\tup a \in \Delta(\alpha_0,\dots,\alpha_{n-1})$ from Lemma~\ref{forks}.
\end{proof}

In the rest of this chapter we use symbol $[\alpha_0,\dots,\alpha_{n-1}]_{\al
X}$ to denote the commutator $[\alpha_0,\dots,\alpha_{n-1}]$ computed in
the~algebra $\al X$.

\begin{corollary} \label{delta-gives-commutators} 
Let $\al A$, $\al B$ are algebras that share a~universe, a~Mal'cev
operation, and congruences $\alpha_0,\dots,\alpha_{n-1}$. Then
\[
  [\alpha_{i_0},\dots,\alpha_{i_{k-1}}]_{\al A}
    = [\alpha_{i_0},\dots,\alpha_{i_{k-1}}]_{\al B}
\]
for all $k\leq n$ and $0\leq i_0 < \dots < i_{k-1} < n$ if and only if
\(
  \Delta_{\al A}(\alpha_0,\dots,\alpha_{n-1})
    = \Delta_{\al B} (\alpha_0,\dots,\alpha_{n-1})
\).
\end{corollary}

\begin{proof} We will prove the corollary by induction on $n$. The case $n=1$
is trivial; for the induction step suppose that for all congruences
$\beta_0,\dots,\beta_{n-1} \in \Con \al A \cap \Con \al B$ such that the
commutators $[\beta_{i_0},\dots,\beta_{i_{k-1}}]$ agree in $\al A$ and $\al B$
for all $k<n$ and $i_0,\dots,i_{k-1}$ pairwise distinct elements from
$\{0,\dots,n-1\}$, we have
\(
  \Delta_{\al A}(\beta_{i_0},\dots,\beta_{i_{k-1}})
  = \Delta_{\al B} (\beta_{i_0},\dots,\beta_{i_{k-1}})
\).
In particular, we have
$\Delta_{\al A}(\alpha_0,\dots,\alpha_{i-1},\alpha_{i+1},\dots,\alpha_{n-1})
  = \Delta_{\al B}(\alpha_0,\dots,\alpha_{i-1},\alpha_{i+1},\dots,\alpha_{n-1})$
for all $i<n$.
Let $q_n$ be a~common strong $n$-cube operation of $\al A$ and $\al B$ (it
can be derived from the~common Mal'cev operation). From
Lemma~\ref{delta-contains-graph-of-cube-term} we know that whenever $\al X$ is
a~Mal'cev algebra then $\tup a \in
\Delta_{\al X}(\alpha_0,\dots,\alpha_{n-1})$ if and only if
\begin{multline} \label{eq:lem4.2}
  \face_j^0 \tup a \in \Delta_{\al X}
    (\alpha_0,\dots,\alpha_{i-1},\alpha_{i+1},\dots,\alpha_{n-1})
    \text{ for all $j<n$, and }\\
  a_{2^n-1} \equiv_{[\alpha_0,\dots,\alpha_{n-1}]_{\al X}}
    q_n(a_0,\dots,a_{2^n-2}).
\end{multline}
Now, suppose that $\tup a \in \Delta_{\al A}(\alpha_0,\dots,\alpha_{n-1})$ hence
(\ref{eq:lem4.2}) is valid for $\al X = \al A$. But since the operation $q_n$,
the commutators $[\alpha_0,\dots,\alpha_{n-1}]$, and the relations
$\Delta(\alpha_0,\dots,\alpha_{i-1},$ $\alpha_{i+1},\dots,\alpha_{n-1})$ agree in
$\al A$ and $\al B$, we get that (\ref{eq:lem4.2}) is also true for $\al X = \al
B$, hence $a \in \Delta_{\al B}(\alpha_0,\dots,\alpha_{n-1})$. This shows that
\(
  \Delta_{\al A}(\alpha_0,\dots,\alpha_{n-1}) \subseteq
  \Delta_{\al B}(\alpha_0,\dots,\alpha_{n-1}).
\)
The other inclusion is analogous.

For the `only if' part, suppose that $\Delta_{\al A}
(\alpha_0,\dots,\alpha_{n-1}) = \Delta_{\al B}(\alpha_0,\dots,\alpha_{n-1})$.
From Lemma \ref{delta-is-a-congruence}(i) it follows that
\(
  \Delta_{\al A}(\alpha_{i_0},\dots,\alpha_{i_{k-1}}) =
  \Delta_{\al B}(\alpha_{i_0},\dots,\alpha_{i_{k-1}})
\)
for all $\{ i_0, \dots, i_{k-1} \} \subseteq \{0,\dots,n-1\}$ with $i$'s
pairwise distinct; and consequently from Theorem
\ref{description-of-higher-commutators},
$[\alpha_{i_0},\dots,\alpha_{i_{k-1}}]_{\al A} =
 [\alpha_{i_0},\dots,\alpha_{i_{k-1}}]_{\al B}$.
\end{proof}

Finally, we get to the proof of the Theorem
\ref{functions-preserving-n-ary-commutator}. We restate the theorem once again.

\theoremstyle{plain}
\newtheorem*{functions-preserving-n-ary-commutators}
  {Theorem \ref{functions-preserving-n-ary-commutator}}

\begin{functions-preserving-n-ary-commutators}
Let $\al A$ be an~algebra with Mal'cev term $q$, and
$\alpha_0,\dots,\alpha_{n-1}$ be congruences of~$\al A$. Then there exists
a~largest clone $\mathcal C$ on~$A$ containing~$q$ such that it preserves
congruences $\alpha_0,\dots,\alpha_{n-1}$, and all commutators of the form
$[\alpha_{i_0},\dots,\alpha_{i_k}]$ where $0\leq i_0 < \dots < i_{k-1} < n$
agree in $\al A$ and $(A,\mathcal C)$.
\end{functions-preserving-n-ary-commutators}

\begin{proof}
We will show that the largest clone satisfying the required properties is the clone
$\mathcal C$ of all polymorphisms of the relation $\Delta_{\al
A}(\alpha_0,\dots,\alpha_{n-1})$. Obviously $\mathcal C \supseteq \Clo \al A$
which implies that $q \in \mathcal C$ and
\(
  \Delta_{\mathcal C}(\alpha_0,\dots,\alpha_{n-1}) \supseteq 
    \Delta_{\al A}(\alpha_0,\dots,\alpha_{n-1}).
\)
But since the relation $\Delta_{\al A} (\alpha_0,\dots,\alpha_{n-1})$ is
compatible with
$\mathcal C$, and the relation $\Delta_{\mathcal C}(\alpha_0,\dots,\alpha_{n-1})$ is the
smallest compatible relation containing $\cube_i^n(a,b)$ for all
$a\equiv_{\alpha_i} b$ and $i<n$, we have 
\(
  \Delta_{\mathcal C}(\alpha_0,\dots,\alpha_{n-1}) =
  \Delta_{\al A}(\alpha_0,\dots,\alpha_{n-1}).
\)

From Corollary~\ref{delta-gives-commutators}, we know that $\mathcal C$
satisfies the specified property. The rest is to prove that $\mathcal C$ is
the largest such clone.  Let $\mathcal B$ be another clone satisfying the property.
Then from the same corollary we get
\(
  \Delta_{\al A}(\alpha_0,\dots,\alpha_{n-1}) =
  \Delta_{\mathcal B}(\alpha_0,\dots,\alpha_{n-1}),
\)
and consequently $\mathcal B \subseteq \Pol (\Delta_{\al A}
(\alpha_0,\dots,\alpha_{n-1})) = \mathcal C$.
\end{proof}
 \section{Proofs of basic properties of higher commutators}

In this chapter we will present alternative proofs of basic properties of
higher commutators formulated in \cite{bul01,am10}. For an arbitrary algebra
$\al A$ and its congruences $\alpha_0,\dots,\alpha_{n-1}$,
$\beta_0,\dots,\beta_{n-1}$, $\gamma$, and $\eta$ the following is
satisfied
\begin{enumerate}
  \item[(HC1)] $[\alpha_0,\dots,\alpha_{n-1}]
    \leq \alpha_0 \meet \dots \meet \alpha_{n-1}$;
  \item[(HC2)] if $\alpha_i\leq \beta_i$ for all $i$ then
    $[\alpha_0,\dots,\alpha_{n-1}] \leq [\beta_0, \dots, \beta_{n-1}]$;
  \item[(HC3)] $[\alpha_0,\dots,\alpha_{n-1}] \leq
    [\alpha_1, \dots, \alpha_{n-1}]$.
\end{enumerate}
Furthermore, if $\al A$ is a~Mal'cev algebra then
\begin{enumerate}
  \item [(HC4)] if $\sigma$ is a~permutation on the set $\{0,\dots,n-1\}$ then
    \[
      [\alpha_0,\dots,\alpha_{n-1}] =
      [\alpha_{\sigma(0)},\dots,\alpha_{\sigma(n-1)}];
    \]
  \item[(HC5)] congruences $\alpha_0,\dots,\alpha_{n-2}$ centralize
    $\alpha_{n-1}$ modulo $\gamma$ if and only if
    \[
      [\alpha_0,\dots,\alpha_{n-1}] \leq \gamma;
    \]
  \item[(HC6)] if $\eta \leq \alpha_0,\dots,\alpha_{n-1}$ then
    \[
      [\alpha_0/\eta,\dots,\alpha_{n-1}/\eta]_{\al A/\eta} =
      ( [\alpha_0,\dots,\alpha_{n-1}]_{\al A} \join \eta ) / \eta;
    \]
  \item[(HC7)] if $I$ is a~non-empty set, and $\rho_i$ are congruences
    of~$\al A$ for all $i\in I$ then
    \[
      \bigjoin_{i\in I} [\alpha_0,\dots,\alpha_{n-2},\rho_i] =
      [\alpha_0,\dots,\alpha_{n-2},\bigjoin_{i\in I} \rho_i];
    \]
  \item[(HC8)] if $i = 1,\dots,n-1$ then
    \(
      [[\alpha_0,\dots,\alpha_{i-1}],\alpha_i,\dots,\alpha_{n-1}] \leq
      [\alpha_0,\dots,\alpha_{n-1}].
    \)
\end{enumerate}

Although properties (HC1--3) can be derived directly from Theorem
\ref{description-of-higher-commutators} and the definition of
$\Delta(\alpha_0,\dots,\alpha_{n-1})$, we will omit these proofs because methods
would work only for Mal'cev algebras; the general case have been proved in
\cite[Proposition~1.3]{mud09}.  We will prove properties (HC4), (HC5), (HC7),
and (HC8)---the last property (HC6) is a~corollary of (HC5).

\begin{proposition}[HC4, {\cite[Proposition 6.1]{am10}}] 
Let $\al A$ be a~Mal'cev algebra and let
$\alpha_0,\dots,\alpha_{n-1} \in \Con \al A$.
Then
 $[\alpha_0,\dots,\alpha_{n-1}] =
  [\alpha_{\sigma(0)},\dots,\alpha_{\sigma(n-1)}]$ for each permutation
  $\sigma$ of~$\{0,\dots,n-1\}$.
\end{proposition}

\begin{proof}
We claim that the relations
$\Delta(\alpha_{\sigma(0)},\dots,\alpha_{\sigma(n-1)})$  and
$\Delta(\alpha_0,\dots,\alpha_{n-1})$ are identical up to permuting coordinates;
precisely
\[
  \Delta(\alpha_{\sigma(0)},\dots,\alpha_{\sigma(n-1)}) =
  \Delta(\alpha_0,\dots,\alpha_{n-1})^{\sigma'}
\]
where $\sigma'$ is defined by
\(
  \sigma'(\tup k) = k_{\sigma^{-1}(0)}\dots k_{\sigma^{-1}(n-1)}
\).
One can check this fact by observing that $\sigma'(\tup k) (\sigma(i)) =
k_{i}$, and consequently
\begin{multline*}
\Delta(\alpha_{\sigma(0)},\dots,\alpha_{\sigma(n-1)})
= \bigjoin_{i<n} \cube_i^n ( \alpha_{\sigma(i)} )
= \bigjoin_{i<n}
  \left( \cube_{\sigma(i)}^n (\alpha_{\sigma(i)}) \right)^{\sigma'} \\
= \bigl( \bigjoin_{i<n} \cube_{i,n}(\alpha_{i}) \bigr)^{\sigma'}
= \Delta(\alpha_0,\dots,\alpha_{n-1})^{\sigma'}.
\end{multline*}
Finally, $\sigma'(11\dots 1) = 11\dots 1$, so
the statement is true from Theorem \ref{description-of-higher-commutators} and
Lemma \ref{forks-of-projections}(ii). 
\end{proof}

\begin{proposition}[HC5, {\cite[Lemma 6.2]{am10}}] 
Let $\al A$ be a~Mal'cev algebra and $\alpha_0,\dots,\alpha_{n-1}$, $\gamma$ be
congruences of $\al A$. Then $\alpha_0,\dots,\alpha_{n-2}$ centralizes
$\alpha_{n-1}$ modulo $\gamma$ if and only if
\(
  \gamma \geq [\alpha_0,\dots,\alpha_{n-1}].
\)
\end{proposition}

\begin{proof}
The `only if' part is given by the definition of the commutator, to prove the
`if' part suppose that $\tup a\in \Delta(\alpha_0,\dots,\alpha_{n-1})$ such that
$a_{k_0\dots k_{n-2}0} \equiv_\gamma a_{k_0\dots k_{n-2}1}$ for all
$k\in2^{n-1} \setminus \{1\dots 1\}$. We want to prove that
$a_{11\dots10} \equiv_\gamma a_{11\dots 1}$. By Lemma
\ref{delta-contains-graph-of-cube-term} we know that
\[
  a_{11\dots 1} \equiv_{[\alpha_0,\dots,\alpha_{n-1}]}
    q_n(a_{00\dots 0},\dots,a_{01\dots 1})
\]
but the right hand side is modulo $\gamma$ congruent to
\[
  q_n(a_{00\dots 0},\dots,a_{01\dots 10},a_{1\dots 10},
      a_{00\dots 0},\dots,a_{01\dots 10}) = a_{1\dots 10}.
\]
So, we know that $a_{1\dots 11} \equiv a_{1\dots 10}$ modulo
$\gamma$ since $\gamma \geq [\alpha_0,\dots,\alpha_{n-1}]$. And finally,
$\alpha_0,\dots,\alpha_{n-2}$ centralizes $\alpha_{n-1}$ modulo $\gamma$ from
Lemma~\ref{term-condition}.
\end{proof}

The condition (HC6) is a~direct corollary of~condition (HC5); for a~proof see
\cite[Corollary 6.3]{am10}. The following two lemmata prepare for the proof
of~(HC7).

\begin{lemma} \label{join-of-deltas} 
Let $\al A$ be an~algebra, $I$ a~non-empty set, $\rho_i \in \Con \al A$
for all $i\in I$, and $\alpha_0,\dots,\alpha_{n-2} \in \Con \al A$. Then
\[
  \Delta(\alpha_0,\dots,\alpha_{n-2},\bigjoin_{i \in I} \rho_i) =
    \bigjoin_{i \in I} \Delta(\alpha_0,\dots,\alpha_{n-2},\rho_i).
\]
\end{lemma}

\begin{proof} The statement can be derived directly from the definition
of~the relation $\Delta(\alpha_0,\dots,\alpha_{n-1})$ by a~simple calculation:
\begin{align*}
  \Delta(\alpha_0,\dots,\alpha_{n-2},\bigjoin_{i \in I} \rho_i) &=
    \Bigl( \cube_0^n\alpha_0 \join \dots
      \join \cube_{n-2}^n \alpha_{n-2}
      \join \cube_{n-1}^n \bigjoin_{i \in I} \rho_i \Bigr) \\
    &= \bigjoin_{i \in I} 
    \bigl( \cube_0^n \alpha_0 \join \dots
      \join \cube_{n-2}^n \alpha_{n-2}
      \join \cube_{n-1}^n \rho_i \bigr) \\
    &= \bigjoin_{i \in I} \Delta(\alpha_0,\dots,\alpha_{n-2},\rho_i). \qedhere
\end{align*}
\end{proof}

The following lemma was in fact proved during the proof
of~\cite[Lemma~6.4]{am10}.

\begin{lemma} \label{commutator-is-continuous} 
Let $\al A$ be an~algebra, $I$ a~non-empty set, $\rho_i \in \Con \al A$
for all $i\in I$, and $\alpha_0,\dots,\alpha_{n-2} \in \Con \al A$. Then
\[
  [\alpha_0,\dots,\alpha_{n-2},\bigjoin_{i \in I} \rho_i] = 
    \bigjoin_{\substack{\{i_0,\dots,i_{k-1}\} \subseteq I, \\ k < \infty}} 
      [\alpha_0,\dots,\alpha_{n-2},
        \bigjoin_{i\in \{i_0,\dots,i_{k-1}\}} \rho_i].
\]
\end{lemma}

\begin{proof}
Throughout the proof we will extensively use compactness of subuniverses of some
fixed algebra that is if $\al
A_i$ for $i\in J$ are subalgebras of some algebra $\al A$, and $a \in
\bigjoin_{i\in J} \al A_i$ then there exists a~finite set $K\subseteq J$ such
that $a \in \bigjoin_{i\in K} \al A_i$.
To shorten the notation let $\eta$ denotes the right hand side of the statement.
Hence
\[
  \eta = \bigjoin_{\substack{\{i_0,\dots,i_{k-1}\} \subseteq I, \\ k < \infty}}
    [\alpha_0,\dots,\alpha_{n-2},\bigjoin_{i\in \{i_0,\dots,i_{k-1}\}} \rho_i].
\]
First, we prove that $\alpha_0$, \dots, $\alpha_{n-2}$
centralize $\bigjoin_{i\in I} \rho_i$ modulo $\eta$.

\begin{claim}
  If $\tup a \in \Delta(\alpha_0,\dots,\alpha_{n-2},\bigjoin_{i\in I} \rho_i)$
  then there is a~finite set $S \subseteq I$ such that $\tup a \in
  \Delta(\alpha_0,\dots,\alpha_{n-2},\bigjoin_{i\in S} \rho_i)$.
\end{claim}

The claim follows from Lemma~\ref{join-of-deltas} and the note at the beginning
of this proof.

\begin{claim} If $a \equiv_\eta b$ then there is a~finite set
$T\subseteq I$ such that $a$ and $b$ are congruent modulo
$[\alpha_0,\dots,\alpha_{n-2}, \bigjoin_{i\in T} \rho_i]$.
\end{claim}

Again, there are finite sets $T_0$, \dots, $T_{k-1}$ such that $a$ and $b$ are
congruent modulo $\bigjoin_{j=0}^{k-1}
[\alpha_0,\dots,\alpha_{n-2},\bigjoin_{i \in T_j} \rho_i]$. And, 
\[
  \bigjoin_{j=0}^{k-1}
    [\alpha_0,\dots,\alpha_{n-2},\bigjoin_{i \in T_j} \rho_i]
    \leq [\alpha_0,\dots,\alpha_{n-2}, \bigjoin_{\rho \in T} \rho]
\]
where $T = \bigcup_{j=0}^{k-1} T_j$. Which completes the proof of the second claim.

Suppose that
$\tup a\in \Delta(\alpha_0,\dots,\alpha_{n-2},\bigjoin_{i\in I} \rho_i)$
and
$a_{k_0\dots k_{n-2}0} \equiv_\eta a_{k_0\dots k_{n-2}1}$
for all $\tup k\in 2^{n-1} \setminus \{1\dots 1\}$. Let $S$ be a~finite set from Claim~1, and let
  $T_{\tup k}$ be finite sets such that $a_{k_0\dots k_{n-2}0}$ and $a_{k_0\dots
  k_{n-2}1}$ are congruent modulo
$[\alpha_0,\dots,\alpha_{n-2},\bigjoin_{\rho\in T_{\tup k}} \rho]$; such sets exist
by Claim~2. Let $U = S \cup T_{0\dots 0} \cup \dots \cup T_{1\dots 10}$ (note that
$U$ is a~finite set) and $\eta' = \bigjoin_{i\in U} \rho_i$. Then
\begin{enumerate}
\item $\tup a \in \Delta(\alpha_0,\dots,\alpha_{n-2},\eta')$, and
\item $a_{k_0\dots k_{n-2}0} \equiv_{[\alpha_0,\dots,\alpha_{n-1},\eta']}
a_{k_0\dots k_{n-2}1}$ for all $k\in 2^{n-1} \setminus \{1\dots 1\}$.
\end{enumerate}
So, from the Lemma \ref{term-condition} we know that $a_{1\dots 10}
\equiv_{[\alpha_0,\dots,\alpha_{n-1},\eta']} a_{1\dots 11}$.  Finally,
$[\alpha_0,\dots,\alpha_{n-1},\eta'] \leq [\alpha_0,\dots,\alpha_{n-1},\eta]$
because $\eta' \leq \eta$. Hence $\alpha_0,\dots,\alpha_{n-2}$ centralize
$\bigjoin_{i\in I} \rho_i$ modulo $\eta$, and consequently
$[\alpha_0,\dots,\alpha_{n-2},\bigjoin_{i\in I} \rho_i] \leq \eta$.

The other inclusion is obvious from the fact that
\[
  [\alpha_0,\dots,\alpha_{n-2},\bigjoin_{i\in I} \rho_i] \geq 
      [\alpha_0,\dots,\alpha_{n-2},\bigjoin_{i\in J} \rho_{i}]
\]
for every finite set $J\subseteq I$. 
\end{proof}

\begin{lemma}[{\cite[Corollary 6.6]{am10}}] \label{hc7-finite-case}
Let $\al A$ be a~Mal'cev algebra and $\alpha_1,\dots,\alpha_{n-1}$, $\rho_1,\dots,\rho_k$ congruences of $\al A$. Then
\[
  [\alpha_0,\dots,\alpha_{n-2},\bigjoin_{i=1}^k \rho_i] =
  \bigjoin_{i=1}^k [\alpha_0,\dots,\alpha_{n-2},\rho_i].
\]
\end{lemma}

\begin{proof} It suffices to prove the statement just for $k=2$. We will write
\[
  (a_0,\dots,a_{2^{n-1}-1}) \equiv_{\Delta(\alpha_0,\dots,\alpha_{n-2},\rho_i)}
  (b_0,\dots,b_{2^{n-1}-1}) 
\]
if $(a_0,\dots,a_{2^{n-1}-1}, b_0,\dots,b_{2^{n-1}-1}) \in
\Delta(\alpha_0,\dots,\alpha_{n-2},\rho_i)$. Note that from Lemma
\ref{delta-is-a-congruence}, we know that the binary relation
\(
  \{ (\tup a,\tup b) : \tup a
    \equiv_{\Delta(\alpha_0,\dots,\alpha_{n-2},\rho_i)} \tup b \}
\)
is a~congruence on $\Delta(\alpha_0,\dots,\alpha_{n-2})$. From
Lemma~\ref{join-of-deltas}, we know that
\[
\Delta(\alpha_0,\dots,\alpha_{n-2},\rho_1\join \rho_2) = 
  \Delta(\alpha_0,\dots,\alpha_{n-2},\rho_1) \join
  \Delta(\alpha_0,\dots,\alpha_{n-2},\rho_2).
\]
Since in Mal'cev algebras $\alpha\circ\beta = \Sg(\alpha\cup\beta)$ for
any pair of congruences $\alpha$, $\beta$, we have that for all $\tup a, \tup
b \in \Delta(\alpha_0,\dots,\alpha_{n-2})$,
\(
  \tup a \equiv_{\Delta(\alpha_0,\dots,\alpha_{n-2},\rho_1 \join \rho_2)}
    \tup b
\)
if and only if there exists $\tup c$ such that
$\tup a \equiv_{\Delta(\alpha_0,\dots,\alpha_{n-2},\rho_1)} \tup c$ and
$\tup b \equiv_{\Delta(\alpha_0,\dots,\alpha_{n-2},\rho_2)} \tup c$.

Now, we prove that
\begin{equation} \label{eq:hc7-fin}
  [\alpha_0,\dots,\alpha_{n-2},\rho_1 \join \rho_2] \leq
  [\alpha_0,\dots,\alpha_{n-2},\rho_1]
    \join [\alpha_0,\dots,\alpha_{n-2},\rho_2].
\end{equation}
Suppose that $a$ and $b$ are congruent modulo
$[\alpha_0,\dots,\alpha_{n-2},\rho_1\join \rho_2]$, hence from
Proposition~\ref{commutator-and-delta} there are $e_0,\dots,e_{2^{n-1}-2}$
such that
\[
  (e_0,\dots,e_{2^{n-1}-2},a)
    \equiv_{\Delta(\alpha_0,\dots,\alpha_{n-2},\rho_1\join \rho_2)}
  (e_0,\dots,e_{2^{n-1}-2},b).
\]
From the above observation, we know that there is a~tuple
$\tup c \in \Delta(\alpha_0,\dots,\alpha_{n-2})$ such that
\begin{equation} \label{eq:one}
  \tup c \equiv_{\Delta(\alpha_0,\dots,\alpha_{n-2},\rho_1)}
    (e_0,\dots,e_{2^{n-1}-2},a)
\end{equation}
and
\begin{equation} \label{eq:two}
  \tup c \equiv_{\Delta(\alpha_0,\dots,\alpha_{n-2},\rho_2)}
    (e_0,\dots,e_{2^{n-1}-2},b)
\end{equation}
If we use Lemma~\ref{delta-contains-graph-of-cube-term} for (\ref{eq:one}), we
get that
\[
  a \equiv_{[\alpha_0,\dots,\alpha_{n-2},\rho_1]}
    q_n(c_0,\dots,c_{2^{n-1}-1},e_0,\dots,e_{2^{n-1}-2}); 
\]
similarly for (\ref{eq:two}), we get that
\[
  b \equiv_{[\alpha_0,\dots,\alpha_{n-2},\rho_2]}
    q_n(c_0,\dots,c_{2^{n-1}-1},e_0,\dots,e_{2^{n-1}-2}). 
\]
Altogether, $a$ and $b$ are congruent modulo
$[\alpha_0,\dots,\alpha_{n-2},\rho_1]
  \join [\alpha_0,\dots,\alpha_{n-2},\rho_2]$.
Which completes the proof
of~(\ref{eq:hc7-fin}). The other inclusion is given by (HC2).
\end{proof}

\begin{proposition}[HC7, {\cite[Lemma 6.7]{am10}}]
  Let $\al A$ be a~Mal'cev algebra with congruences
  $\alpha_0,\dots,\alpha_{n-2}$, and $\rho_i$, $i\in I$ for $I$ non-empty set.
  Then
  \[
    [\alpha_0,\dots,\alpha_{n-2},\bigjoin_{i\in I} \rho_i]
    = \bigjoin_{i\in I} [\alpha_0,\dots,\alpha_{n-2},\rho_i].
  \]
\end{proposition}

\begin{proof}
If $I$ is a~finite set then the proposition is given by
Lemma~\ref{hc7-finite-case}. If $I$ is infinite, we can first use
Lemma~\ref{commutator-is-continuous} to get
\[
  [\alpha_0,\dots,\alpha_{n-2},\bigjoin_{i\in I} \rho_i]
  = \bigjoin_{\{i_0,\dots,i_{k-1}\} \subseteq I}
    [\alpha_0,\dots,\alpha_{n-2},
      \bigjoin_{i\in \{i_0,\dots,i_{k-1}\}} \rho_i].
\]
Then by using the finite case, the right hand side is equal to
\[
  \bigjoin_{\substack{{\{i_0,\dots,i_{k-1}\} \subseteq I}, \\
  {i\in \{i_0,\dots,i_{k-1}\}}}}
    [\alpha_0,\dots,\alpha_{n-2}, \rho_i]
    = \bigjoin_{i\in I} [\alpha_0,\dots,\alpha_{n-2}, \rho_i]. \qedhere
\]
\end{proof}

\begin{proposition}[HC8, {\cite[Proposition 6.14]{am10}}]
Let $\al A$ be a~Mal'cev algebra with congruences
$\alpha_0,\dots,\alpha_{n-1}$, and $i \in \{1,\dots,n-1\}$. Then
\[
  [[\alpha_0,\dots,\alpha_{i-1}],\alpha_i,\dots,\alpha_{n-1}] \leq
    [\alpha_0,\dots,\alpha_{n-1}].
\]
\end{proposition}

\begin{proof}
For $m\geq i$, define the map $e_m\colon 2^m \to 2^{m-i+1}$ by
\(
  \tup k \mapsto k'k_{i}\dots k_{m-1}
\)
where $k' = k_0 \cdot k_1 \cdot \hdots \cdot k_{i-1}$.
We claim that
\begin{equation} \label{eq:hc8}
  \Delta([\alpha_0,\dots,\alpha_{i-1}],\alpha_i,\dots,\alpha_{m-1}) ^ {e_m}
    \leq \Delta(\alpha_0,\dots,\alpha_{m-1}).
\end{equation}
Because
$\Delta([\alpha_0,\dots,\alpha_{i-1}],\alpha_i,\dots,\alpha_{m-1})^{e_m}$ is
clearly a~subuniverse of $\al A^{\!2^m}$ generated by the set
\[
  (\cube_0^{m-i+1} [\alpha_0,\dots,\alpha_{i-1}]) ^ {e_m}
    \cup \bigcup_{j<m-i} (\cube_{j+1}^{m-i+1} \alpha_{i+j}) ^ {e_m},
\]
it suffices to prove that this is a~subset of
$\Delta(\alpha_0,\dots,\alpha_{m-1})$.  The inclusions
\[
  (\cube_{j+1}^{m-i+1} \alpha_{i+j}) ^ {e_m}
    \subseteq \Delta(\alpha_0,\dots,\alpha_{n-1})
\]
are consequences of the fact that $e_m(\tup k)(j+1) = k_{j+i}$ for all
$j<m-i$.  The other inclusion, $(\cube_0^{m-i+1}[\alpha_0,\dots,\alpha_{i-1}])
^ {e_m} \subseteq \Delta(\alpha_0,\dots,\alpha_{m-1})$, can be proved by
induction on $m$. If $m=i$ then
\[
  (\cube_0^1 [\alpha_0,\dots,\alpha_{i-1}]) ^ {e_m}
  = \{ (a,\dots,a,b) \midbar a \equiv_{[\alpha_0,\dots,\alpha_{i-1}]} b \},
\]
and all the elements of~this set are in~$\Delta(\alpha_0,\dots,\alpha_{i-1})$
from Proposition~\ref{commutator-and-delta}. For the induction step, observe
that
\begin{multline*}
  \face_m^0 \bigl((\cube_0^{(m+1)-i+1} (a,b)) ^ {e_{m+1}}\bigr)
    = \face_m^1 \bigl((\cube_0^{(m+1)-i+1} (a,b)) ^ {e_{m+1}}\bigr) \\
    = \left(\cube_0^{m-i+1} (a,b)\right) ^ {e_{m}}.
\end{multline*}
So the inclusion follows from Lemma~\ref{delta-is-a-congruence}(ii).
Finally, from (\ref{eq:hc8}) for $m=n$, the fact that $e_n(\tup k) = 1\dots 1$
if and only if $\tup k = 1\dots 1$, and Lemma \ref{forks-of-projections} we get
\[
  \psi_{1\dots 1}
    (\Delta([\alpha_0,\dots,\alpha_{i-1}],\alpha_i,\dots,\alpha_{n-1}))
  \leq \psi_{1\dots 1} (\Delta(\alpha_0,\dots,\alpha_{n-1})).
\]
The rest is Theorem \ref{description-of-higher-commutators}.
\end{proof}
 
\subsection*{Acknowledgements}

The author would like to thank Erhard Aichinger for many hours of fruitful
discussions and many helpful comments on the manuscript.

\end{document}